\documentclass[12pt,reqno]{amsart}
\usepackage{amsfonts}
\usepackage{amssymb,color}
\usepackage{amssymb}

\usepackage[colorlinks=true]{hyperref}
\hypersetup{urlcolor=blue, citecolor=blue}

\usepackage[margin=3cm, a4paper]{geometry}

\newtheorem{theo}{Theorem}[section]
\newtheorem{lemm}[theo]{Lemma}
\newtheorem{defi}[theo]{Definition}

\newtheorem{rema}[theo]{Remark}
\numberwithin{equation}{section}

\newcommand{\bal}{\begin{align}}
\newcommand{\bbal}{\begin{align*}}
\newcommand{\beq}{\begin{equation}}
\newcommand{\eeq}{\end{equation}}
\newcommand{\bca}{\begin{cases}}
\newcommand{\eca}{\end{cases}}

\newcommand{\pa}{\partial}

\newcommand{\na}{\nabla}
\newcommand{\De}{\Delta}

\newcommand{\cd}{\cdot}

\newcommand{\dd}{\mathrm{d}}

\newcommand{\R}{\mathbb{R}}

\newcommand{\D}{\mathrm{div}}
\newcommand{\Z}{\mathbb{Z}}

\begin{document}

\subjclass[2010]{76W05}
\keywords{The non-resistive MHD equations, non-uniform  continuous dependence, continuous dependence}

\title[Non-resistive MHD equations]{Non-uniform  continuous dependence and continuous dependence for the non-resistive MHD equations}

\author[J. Li]{Jinlu Li}
\address{School of Mathematics and Computer Sciences, Gannan Normal University, Ganzhou 341000, China}
\email{lijl29@mail2.sysu.edu.cn}

\author[Z. Yin]{Zhaoyang Yin}
\address{Department of Mathematics, Sun Yat-sen University, Guangzhou, 510275, China \& Faculty of Information Technology,
 Macau University of Science and Technology, Macau, China}
\email{mcsyzy@mail.sysu.edu.cn}

\author[W. Zhu]{Weipeng Zhu}
\address{School of Mathematics and Information Science, Guangzhou University, Guangzhou, 510006, China}
\email{mathzwp2010@163.com}

\begin{abstract}
In this paper, we establish the continuous dependence for the non-resistive MHD equations in Sobolev spaces. Our obtained result fills considerably the recent result \cite{F} for the continuous dependence on initial data. We also show that this result is optimal. More precisely, we proved that the data-to-solution map is not uniform  continuous dependence in Sobolev spaces.
\end{abstract}

\maketitle

\section{Introduction and main result}

In the paper, we consider the following Cauchy problem of the non-resistive MHD equations:
\begin{equation}\label{MHD}\begin{cases}
\partial_tu+u\cdot\nabla u-\Delta u+\nabla P=b\cdot \nabla b, \\
\partial_tb+u\cdot \nabla b=b\cdot\nabla u,\\
\mathrm{div} u=\mathrm{div} b=0,\quad (u,b)|_{t=0}=(u_0,b_0),
\end{cases}\end{equation}
where the unknowns are the vector fields  $u: \R\times \R^d\rightarrow \R^d$,  $b: \R\times \R^d\rightarrow \R^d$ and the scalar function $P$. Here, $u$ and $b$ are the velocity and magnetic, respectively, while $P$ denotes the pressure.

In \cite{Jiu06}, Jiu and Niu were first concerned with the well-posedness of system \eqref{MHD} and established the local existence of solutions in 2D for initial data in $H^s,\ s\geq 3$. Fefferman et al. obtained local-in-time existence of strong solutions to \eqref{MHD} in $\R^d,\ d=2,3$ with the initial data $(u_0,b_0)\in H^s(\mathbb{R}^d)\times H^s(\mathbb{R}^d),\ s>\frac d2$ in \cite{F} and $(u_0,b_0)\in H^{s-1+\varepsilon}(\mathbb{R}^d)\times H^s(\mathbb{R}^d),\ s>\frac d2,\ 0<\varepsilon<1$ in \cite{F1}. Chemin et al. in \cite{C} presented the local existence of solutions to \eqref{MHD} in $\R^d,\ d=2,3$ with the initial data $(u_0,b_0)\in B^{\frac d2-1}_{2,1}(\mathbb{R}^d)\times B^{\frac d2}_{2,1}(\mathbb{R}^d)$ and also proved the corresponding solution is unique in 3D case.  Wan in \cite{W} resolved the uniqueness of the solution in the 2D case by using mixed space-time Besov spaces. Recently, Li, Tan and Yin in \cite{L.T.Y} obtain the local existence and uniqueness in the homogeneous Besov spaces $(u_0,b_0)\in\dot{B}^{\frac dp-1}_{p,1}(\mathbb{R}^d)\times \dot{B}^{\frac dp}_{p,1}(\mathbb{R}^d)$ where $p\in[1,2d]$.
%For other results concerning regularity criteria, we refer to \cite{F-T,Z.F}

For the global existence of small solutions to the system \eqref{MHD}, there are many interesting results when the background magnetic filed is $\vec{e}_1$, i.e. in the case that the initial data $(u_0,b_0)$ is sufficiently close to the equilibrium state $(\vec{0},\vec{e}_1)$, we refer the reader to see \cite{Abi17,Hu14,Lin15,LinZ15,Ren14,Xu15,Zha14}. In \cite{Zha16}, Zhang obtained the global well-posedness result with the large initial data under the assumption that the background magnetic field is large enough.  Subsequently, Zhai and Zhang \cite{Zhai16} considered the global existence and uniqueness of solution to system \eqref{MHD} with non-equilibrium background magnetic field.

However, the continuous dependence property of solutions for the Cauchy problem of the non-resistive MHD equations with initial data $(u_0,b_0)\in H^s(\mathbb{R}^d)\times H^s(\mathbb{R}^d),\ s>\frac d2$ has not been studied yet. Our first goal in the paper is to prove the continuous dependence property of solutions to the non-resistive MHD equations in such Sobolev spaces. Together with recent local existence and uniqueness results in \cite{F,F1}, we establish local well-posedness of the Cauchy problem of the non-resistive MHD system in the sense of Hadamard.

In order to state our main result, we first recall the following local-in-time existence of strong solutions to \eqref{MHD} in \cite{F}:
\begin{lemm}\cite{F}\label{le0}
For $s>\frac d2,\ d=2,3$, and initial data $u_0,b_0\in H^s(\R^d)$ with $\D u_0=\D b_0=0$, there exists a time $T=T(s,||u_0||_{H^s},||b_0||_{H^s})>0$ such that the system \eqref{MHD} have a unique solution $(u,b)$, with $u,b\in \mathcal{C}([0,T];H^s)$ and $\na u\in L^2_T(H^s)$. Moreover, for all $t\in[0,T]$, there holds
\[||u(t)||^2_{H^s}+||b(t)||^2_{H^s}\leq \frac{||u_0||^2_{H^s}+||b_0||^2_{H^s}}{1-C_sT(||u_0||^2_{H^s}+||b_0||^2_{H^s})},\]
\[\int^t_0||\nabla u(\tau)||^2_{H^s}\dd \tau\leq C_s\frac{||u^2_0||_{H^s}+||b_0||^2_{H^s}}{1-C_sT(||u_0||^2_{H^s}+||b_0||^2_{H^s})}.\]
\end{lemm}

To solve the continuous dependence property of the system \eqref{MHD}, the main difficulty is that the system is only partially parabolic, owing to the magnetic equation which is of hyperbolic type. That is if $(u^1,b^1)$ and $(u^2,b^2)$ are two solutions of the system \eqref{MHD}, then we set $\delta b=b^1-b^2$ satisfies
\[\partial_t\delta b+u^1\cdot\nabla \delta b=\delta b\cd\na u^1+b^2\cd \na \delta u-\delta u\cdot\nabla b^2.\]
Due to the term $\na u^2\in \mathcal{C}([0,T];H^{s-1})$, this precludes any attempt to tackle with the estimate $||\delta b||_{H^s}$. Inspired by Bona-Smith \cite{BS} and Guo-Li-Yin \cite{G}, we can take $(u^2,b^2)$ as the solution with initial data $S_j(u_0,b_0)$ to overcome this problem. Therefore,  we first will prove the continuous dependence property of solutions to the system \eqref{MHD}.

Our first main theorem can be stated as follows:
\begin{theo}\label{th1}
Let $d=2,3$ and $s>\frac d2$. Denote $\bar{\mathbb{N}}=\mathbb{N}\cup\{\infty\}$. Assume that $(u^n,b^n)_{n\in \bar{\mathbb{N}}}$ be the solution to the system \eqref{MHD} with initial data $(u^n_0,b^n_0)_{n\in \bar{\mathbb{N}}}$. If $(u^n_0,b^n_0)$ tends to $(u^\infty_0,b^\infty_0)$ in $H^s$, then there exists a positive time $T$ independent of $n$ such that $(u^n,b^n)$ tends to $(u^\infty,b^\infty)$ in $\mathcal{C}([0,T];H^s)\times \mathcal{C}([0,T];H^s)$.
\end{theo}

When $b=0$, System \eqref{MHD} becomes Navier-Stokes euqations. It is well known that the solution map of Navier-Stokes euqation is local lipshitz. That is, if $||u||_{H^s},||v||_{H^s}\leq R$, we can obtain that
\bbal
||u(t)-v(t)||_{H^s}\leq C||u_0-v_0||_{H^s}, \qquad t\in [0,T],
\end{align*}
for some positive time $T$ depending only on $s$ and $R$.

However, the magnetic equation of the non-resistive MHD equations have the construction of transport diffusion equation. Therefore, motivated by Himonas-Misio{\l}ek \cite{HM}, we can conclude that Cauchy problem (1.3) is not uniformly continuous. The main difficulty is to bound $u^{\omega,n}(t,x)-u^{\omega,n}(0,x)$ less then $n^{-1}$ in Sobolev $H^{s-1}$ norm which is the main part of $\pa_tb^{h,\omega,n}+u^{\omega,n}\cd \na b^{h,\omega,n}$. To overcome it, we split $u^{\omega,n}(t,x)-u^{\omega,n}(0,x)$ into two parts: $u^{\omega,n}(t,x)-e^{t\Delta}u^{\omega,n}(0,x)$ and $e^{t\Delta}u^{\omega,n}(0,x)-u^{\omega,n}(0,x)$. We apply the Duhamel's principle to tackle with the first part. The second part can be seen as $\Delta u^{\omega,n}(0,x)$, and its estimate adds the factor $n^{-2\delta}$ compared with $u^{\omega,n}(0,x)$.

Our second main theorem can be stated as follows:

\begin{theo}\label{th2}
Let $d=2,3$ and $s>\frac d2$. The data-to-solution map for the system \eqref{MHD} is not uniformly continuous from any bounded subset in $H^s$ into $C([0,T];H^s)$. That is, there exists two sequences of solutions $(u^n,b^n)$ and $(v^n,c^n)$ such that
\bbal
&||u^n_0||_{H^s}+||b^n_0||_{H^s}+||v^n_0||_{H^s}+||c^n_0||_{H^s}\lesssim 1,
\\&\lim_{n\rightarrow \infty}\big(||u^n_0-v^n_0||,b^n_0-c^n_0||_{H^s}\big)= 0,
\\&\liminf_{n\rightarrow \infty}\big(||u^n(t)-v^n(t)||_{H^s}+||b^n(t)-c^n(t)||_{H^s}\gtrsim |\sin t|,  \quad t\in[0,T].
\end{align*}
\end{theo}

The paper is organized as follows. In Section 2, we recall the Littlewood-Paley theory and give some properties of Besov spaces. In Section 3,
we introduce some lemmas to overcome the difficulty of the problem. In Section 4, we prove the  continuous dependence property of the system \eqref{MHD}. In Section 5, we show that the solution of the system \eqref{MHD} is non-uniform  continuous dependence.

\vspace*{1em}

\noindent\textbf{Notations.} Given a Banach space $X$, we denote its norm by $\|\cdot\|_{X}$. Since all spaces of functions are over $\mathbb{R}^d$, for simplicity, we drop  $\mathbb{R}^d$ in our notations of function spaces if there is no ambiguity. The symbol $A\lesssim B$ denotes that there exists a constant $c>0$ independent of $A$ and $B$, such that $A\leq c B$. The symbol $A\simeq B$ represents $A\lesssim B$ and $B\lesssim A$.

\section{Littlewood-Paley analysis}

In this section, we will recall some facts about the Littlewood-Paley decomposition, the nonhomogeneous Besov spaces and their some useful properties. For more details, the readers can refer to \cite{B.C.D}.

There exists a couple of smooth functions $(\chi,\varphi)$ valued in $[0,1]$, such that $\chi$ is supported in the ball $\mathcal{B}\triangleq \{\xi\in\mathbb{R}^d:|\xi|\leq \frac 4 3\}$, and $\varphi$ is supported in the ring $\mathcal{C}\triangleq \{\xi\in\mathbb{R}^d:\frac 3 4\leq|\xi|\leq \frac 8 3\}$. Moreover,
$$\forall\,\ \xi\in\mathbb{R}^d,\,\ \chi(\xi)+{\sum\limits_{j\geq0}\varphi(2^{-j}\xi)}=1,$$
$$\forall\,\ 0\neq\xi\in\mathbb{R}^d,\,\ {\sum\limits_{j\in \Z}\varphi(2^{-j}\xi)}=1,$$
$$|j-j'|\geq 2\Rightarrow\textrm{Supp}\,\ \varphi(2^{-j}\cdot)\cap \textrm{Supp}\,\ \varphi(2^{-j'}\cdot)=\emptyset,$$
$$j\geq 1\Rightarrow\textrm{Supp}\,\ \chi(\cdot)\cap \textrm{Supp}\,\ \varphi(2^{-j}\cdot)=\emptyset,$$
Then, we can define the nonhomogeneous dyadic blocks $\Delta_j$ and nonhomogeneous low frequency cut-off operator $S_j$ as follows:
$$\Delta_j{u}= 0,\,\ if\,\ j\leq -2,\quad
\Delta_{-1}{u}= \chi(D)u=\mathcal{F}^{-1}(\chi \mathcal{F}u),$$
$$\Delta_j{u}= \varphi(2^{-j}D)u=\mathcal{F}^{-1}(\varphi(2^{-j}\cdot)\mathcal{F}u),\,\ if \,\ j\geq 0,$$
$$S_j{u}= {\sum\limits_{j'=-\infty}^{j-1}}\Delta_{j'}{u}.$$

\begin{defi}(\cite{B.C.D})\label{de2.3}
Let $s\in\mathbb{R}$ and $1\leq p,r\leq\infty$. The nonhomogeneous Besov space $B^s_{p,r}$ consists of all tempered distribution $u$ such that
\begin{align*}
||u||_{B^s_{p,r}}\triangleq \Big|\Big|(2^{js}||\Delta_j{u}||_{L^p})_{j\in \Z}\Big|\Big|_{\ell^r(\Z)}<\infty.
\end{align*}
\end{defi}
\begin{rema}(\cite{B.C.D})
When $p=r=2$, we have $B^s_{2,2}(\R^d)=H^s(\R^d)$. Here, $H^s(\R^d)$ is the standard Sobolev space with the norm
\bbal
||u||^2_{H^s}:=\int_{\R^d}(1+|\xi^2|)^s|\hat{u}(\xi)|^2\dd \xi.
\end{align*}
 If $s>\frac dp$, we also have $||u||_{L^\infty}\lesssim ||u||_{B^s_{p,r}}$.
\end{rema}

Then, we have the following product laws.
\begin{lemm}(\cite{B.C.D})\label{le1}
Let $s>\frac d2,\ d\geq 2$. Then there exists a constant $C=C(d,s)$ such that
\[||uv||_{H^{s-1}}\leq C\big(||u||_{L^\infty}||v||_{H^{s-1}}+||v||_{L^\infty}||u||_{H^{s-1}}\big),\]
\[||uv||_{H^s}\leq C||u||_{H^s}||v||_{H^s}, \quad ||uv||_{H^{s-1}}\leq C||u||_{H^{s-1}}||v||_{H^s}.\]
\end{lemm}

\begin{lemm}(\cite{B.C.D,L.Y})\label{le2}
Let $\sigma\in\mathbb{R}$ and $1\leq p,r\leq\infty$. Let $v$ be a vector field over $\mathbb{R}^d$. Assume that $\sigma>-d\min\{1-\frac{1}{p},\frac{1}{p}\}$. Define $R_j=[v\cdot\nabla,\Delta_j]f$. There a constant $C=C(p,\sigma,d)$ such that \\
\begin{equation*}
\big|\big|(2^{j\sigma}||R_j||_{L^p})_{j\geq-1}\big|\big|_{\ell^r}\leq
\begin{cases}
 C||\nabla v||_{B^{\frac dp}_{p,\infty}\cap L^\infty}||f||_{B^\sigma_{p,r}},\ \mathrm{if} \quad \sigma<1+\frac{d}{p},\\
 C||\nabla v||_{B^{\frac dp+1}_{p,\infty}}||f||_{B^{\sigma}_{p,r}},\quad \quad \mathrm{if} \quad \sigma=1+\frac{d}{p}\ ,r>1,\\
 C||\nabla v||_{B^{\sigma-1}_{p,r}}||f||_{B^\sigma_{p,r}}, \quad \quad \quad \quad  \mathrm{otherwise}.
\end{cases}\end{equation*}
\end{lemm}

\begin{lemm}(\cite{H-H,HM})\label{le3}
Let $\phi\in S(\mathbb{R})$, $\delta>0$ and $s\geq 0$. Then, we have for all $\alpha\in \mathbb{R}$
\bbal
&\lim\limits_{n\rightarrow \infty}n^{-\frac\delta2}||\phi(\frac{x}{n^{\delta}})||_{H^s(\R)}=||\phi||_{L^2(\R)},
\\&\lim\limits_{n\rightarrow \infty}n^{-\frac12\delta-s}||\phi(\frac{x}{n^{\delta}})\cos(nx-\alpha)||_{H^s(\R)}=\frac{1}{\sqrt{2}}||\phi||_{L^2(\R)},
\\&\lim\limits_{n\rightarrow \infty}n^{-\frac12\delta-s}||\phi(\frac{x}{n^{\delta}})\sin(nx-\alpha)||_{H^s(\R)}=\frac{1}{\sqrt{2}}||\phi||_{L^2(\R)}.
\end{align*}
\end{lemm}

\section{Continuous dependence}

In this section, we will show that the solution of is continuous dependence with initial data. First, we establish some Sobolev norm estimates for smooth solutions of system \eqref{MHD}, which is the key component in the proof of Theorem \ref{th1}. In order to simplify the notation, we define $(u,b)\in X_s(T)$ if
$$(u,b)\in \mathcal{C}([0,T];B^{\sigma}_{2,2}), \qquad \nabla u\in L^2([0,T];B^{\sigma}_{2,2}), \qquad \sigma\in\mathbb{R}.$$.

\begin{lemm}\label{le-estimate}
Let $d=2,3$ and $s>\frac d2$. Suppose that $(u^1,b^1)\in X_{s}(T)$ and $(u^2,b^2)\in X_{s+1}(T)$ are two solutions of \eqref{MHD} with initial data $(u^1_0,b^1_0)$ and $(u^2_0,b^2_0)$ respectively. Denote $\delta u=u^1-u^2$ and $\delta b=b^1-b^2$. Then, we have for all $t\in[0,T]$,
\bbal
&\quad ||\delta u(t)||^2_{B^{s-1}_{2,2}}+||\delta b(t)||^2_{B^{s-1}_{2,2}}+\int^t_0||\nabla\delta u(\tau)||^2_{B^{s-1}_{2,2}}\dd \tau
\\&\leq (||\delta u_0||^2_{B^{s-1}_{2,2}}+||\delta b_0||^2_{B^{s-1}_{2,2}})e^{\mathrm{A}(t)},
\end{align*}
and
\bbal
&\quad||\delta u(t)||^2_{B^{s}_{2,2}}+||\delta b(t)||^2_{B^{s}_{2,2}}+\int^t_0||\nabla\delta u(\tau)||^2_{B^{s}_{2,2}}\dd \tau
\\&\leq \Big(||\delta u_0||^2_{B^{s}_{2,2}}+||\delta b_0||^2_{B^{s}_{2,2}} +C\int^t_0||b^2||^2_{H^{s+1}}(||\delta u||^2_{B^{s-1}_{2,2}}+||\nabla\delta u||^2_{B^{s-1}_{2,2}}) \dd \tau\Big)e^{\mathrm{A}(t)},
\end{align*}
with
\bbal
\mathrm{A}(t)=C\int^t_0(1+||u^1||^2_{H^{s+1}}+||u^2||^2_{H^{s+1}}+||b^1||^2_{H^s}+||b^2||^2_{H^s})\dd \tau.
\end{align*}
\end{lemm}
\begin{proof}
It is easy to show that
\beq\label{MHD1-1}\bca
\partial_t\delta u+u^1\cdot\nabla \delta u+\delta u\cd\nabla u^2-\De \delta u+\nabla P=b^1\cdot \nabla \delta b+\delta b\cd \na b^2, \\
\partial_t\delta b+u^1\cdot \nabla \delta b+\delta u\cd \na b^2=b^1\cdot\nabla \delta u+\delta b\cd \na u^2,\\
\mathrm{div}\delta u=\mathrm{div}\delta b=0,\quad (\delta u,\delta b)|_{t=0}=(\delta u_0,\delta b_0).
\eca\eeq
Now, we apply $\De_j$ to \eqref{MHD1-1}, and take the inner product with $(\De_j\delta u,\De_j\delta b)$ and integrate by parts to have
\bal\label{eq1}
\frac12\frac{\dd}{\dd t}(||\De_j\delta u||^2_{L^2}+||\De_j\delta b||^2_{L^2})+||\De_j\nabla\delta u||^2_{L^2}\leq K_1+K_2+K_3+K_4+K_5,
\end{align}
where
\bbal
&K_1=-\int_{\R^d}[\De_j,u^1\cdot \na]\delta u\cd \De_j\delta u\ \dd x, \quad \quad K_2=-\int_{\R^d}[\De_j,u^1\cdot \na]\delta b\cd \De_j\delta b\ \dd x,\\
&K_3=\int_{\R^d}\De_j(b^1\cdot \nabla \delta b+\delta b\cd \na b^2-\delta u\cd\nabla u^2)\De_j\delta u\ \dd x,\\
&K_4=\int_{\R^d}\De_j(b^1\cdot\nabla \delta u+\delta b\cd \na u^2)\De_j\delta b\ \dd x,\\
&K_5=-\int_{\R^d}\De_j(\delta u\cd\nabla b^2)\De_j\delta b\ \dd x.
\end{align*}
First, we estimate the above terms in the $B^{s-1}_{2,2}$ norm.  According to Lemmas \ref{le1}-\ref{le2}, it is easy to estimate
\bal\label{eq1-1}
\begin{split}
&|K_1|\lesssim ||[\De_j,u^1\cdot \na]\delta u||_{L^2}||\De_j\delta u||_{L^2}\lesssim 2^{-2j(s-1)}c^2_j||\na u^1||_{B^{s}_{2,2}}||\delta u||^2_{B^{s-1}_{2,2}},\\
&|K_2|\lesssim ||[\De_j,u^1\cdot \na]\delta b||_{L^2}||\De_j\delta b||_{L^2}\lesssim 2^{-2j(s-1)}c^2_j||\na u^1||_{B^{s}_{2,2}}||\delta b||^2_{B^{s-1}_{2,2}},\\
&|K_3|\lesssim 2^{-2j(s-1)}c^2_j(||b^1,b^2||_{B^{s}_{2,2}}||\nabla\delta u||_{B^{s-1}_{2,2}}||\delta b||_{B^{s-1}_{2,2}}
\\& \qquad \qquad +||u^2||_{H^s}||\delta u||_{B^{s-1}_{2,2}}||\nabla\delta u||_{B^{s-1}_{2,2}}),\\
&|K_4|\lesssim 2^{-2j(s-1)}c^2_j(||b^1||_{B^{s}_{2,2}}||\nabla\delta u||_{B^{s-1}_{2,2}}||\delta b||_{B^{s-1}_{2,2}}+||\nabla u^2||_{H^s}||\delta b||^2_{B^{s-1}_{2,2}}),
\\&|K_5|\lesssim 2^{-2j(s-1)}c^2_j||b^2||_{B^{s}_{2,2}}(||\delta u||_{B^{s-1}_{2,2}}+||\nabla\delta u||_{B^{s-1}_{2,2}})||\delta b||_{B^{s-1}_{2,2}}.
\end{split}\end{align}
Multiplying the inequality above by $2^{2j(s-1)}$ and summing over $j\geq -1$, we conclude from \eqref{eq1}-\eqref{eq1-1} that
\bbal
&\quad \frac{\dd}{\dd t}(||\delta u||^2_{B^{s-1}_{2,2}}+||\delta b||^2_{B^{s-1}_{2,2}})+||\na \delta u(t)||^2_{B^{s-1}_{2,2}}\nonumber
\\&\leq A'(t)(||\delta u||^2_{B^{s-1}_{2,2}}+||\delta b||^2_{B^{s-1}_{2,2}})+\frac12||\na \delta u||^2_{B^{s-1}_{2,2}},
\end{align*}
which leads to
\bbal
&\quad||\delta u(t)||^2_{B^{s-1}_{2,2}}+||\delta b(t)||^2_{B^{s-1}_{2,2}}+\int^t_0||\nabla \delta u||^2_{B^{s-1}_{2,2}}\dd \tau\\&\leq ||\delta u_0||^2_{B^{s-1}_{2,2}}+||\delta b_0||^2_{B^{s-1}_{2,2}}\nonumber+\int^t_0\mathrm{A}'(\tau)(||\delta u||^2_{B^{s-1}_{2,2}}+||\delta b||^2_{B^{s-1}_{2,2}})\dd \tau.
\end{align*}
This obtain the first result of this lemma. Now, we estimate the above terms in the $B^{s}_{2,2}$ norm. For the terms of $K_i,i=1,2,3,4,5$, we also infer from Lemmas \ref{le1}-\ref{le2} that
\bal\label{eq2-1}\begin{split}
&|K_1|\lesssim ||[\De_j,u^1\cdot \na]\delta u||_{L^2}||\De_j\delta u||_{L^2}\lesssim 2^{-2js}c^2_j||\na u^1||_{B^{s}_{2,2}}||\delta u||^2_{B^{s}_{2,2}},\\
&|K_2|\lesssim ||[\De_j,u^1\cdot \na]\delta b||_{L^2}||\De_j\delta b||_{L^2}\lesssim 2^{-2js}c^2_j||\na u^1||_{B^{s}_{2,2}}||\delta b||^2_{B^{s}_{2,2}},\\
&|K_3|\lesssim 2^{-2js}c^2_j(||b^1,b^2||_{B^{s}_{2,2}}||\nabla\delta u||_{B^{s}_{2,2}}||\delta b||_{B^{s}_{2,2}}+||u^2||_{H^s}||\delta u||_{B^{s}_{2,2}}||\nabla\delta u||_{B^{s}_{2,2}}),\\
&|K_4|\lesssim 2^{-2js}c^2_j(||b^1||_{B^{s}_{2,2}}||\nabla\delta u||_{B^{s}_{2,2}}||\delta b||_{B^{s}_{2,2}}+||\nabla u^2||_{H^s}||\delta b||^2_{B^{s}_{2,2}}),\\
&|K_5|\lesssim 2^{-2js}c^2_j(||\delta u,\nabla\delta u||_{B^{s-1}_{2,2}}||b^2||_{H^{s+1}})||\delta b||_{B^{s}_{2,2}}.
\end{split}\end{align}
%For $K_5$, we get from Lemma \ref{le1} that
%\bal\label{eq2-2}\begin{split}
%|K_5|&\lesssim 2^{-2js}c^2_j(||\delta u\otimes b^2||_{H^{s+1}})||\nabla \delta b||_{B^{s-1}_{2,2}}
%\\&\lesssim 2^{-2js}c^2_j(||\delta u||_{H^{s+1}}||b^2||_{H^s}+||\delta u||_{B^{s}_{2,2}}||b^2||_{H^{s+1}})||\delta b||_{B^{s}_{2,2}}
%\\&\lesssim 2^{-2js}c^2_j(||\delta u,\nabla\delta u||_{B^{s}_{2,2}}||b^2||_{H^s}+||\delta u,\nabla\delta u||_{B^{s-1}_{2,2}}||b^2||_{H^{s+1}})||\delta b||_{B^{s}_{2,2}}
%\end{split}\end{align}
Multiplying the inequality above by $2^{2j{s}}$ and summing over $j\geq -1$, we deduce from \eqref{eq1}, \eqref{eq2-1} that
\bbal
&\quad \frac{\dd}{\dd t}(||\delta u||^2_{B^{s}_{2,2}}+||\delta b||^2_{B^{s}_{2,2}})+||\na \delta u(t)||^2_{B^{s}_{2,2}}\nonumber
\\&\leq A'(t)(||\delta u||^2_{B^{s}_{2,2}}+||\delta b||^2_{B^{s}_{2,2}})+\frac12||\na \delta u||^2_{B^{s}_{2,2}}
\\&\quad +||b^2||^2_{H^{s+1}}(||\delta u||^2_{B^{s-1}_{2,2}}+||\nabla\delta u||^2_{B^{s-1}_{2,2}}),
\end{align*}
which implies
\bbal\begin{split}
&||\delta u(t)||^2_{B^{s}_{2,2}}+||\delta b(t)||^2_{B^{s}_{2,2}}+\int^t_0||\na \delta u(\tau)||^2_{B^{s}_{2,2}}\dd \tau
\\&\leq ||\delta u_0||^2_{B^{s}_{2,2}}+||\delta b_0||^2_{B^{s}_{2,2}}+\int^t_0\mathrm{A}'(\tau)(||\delta u||^2_{B^{s}_{2,2}}+||\delta b||^2_{B^{s}_{2,2}})\dd \tau
\\& \quad +C\int^t_0||b^2||^2_{H^{s+1}}(||\delta u||^2_{B^{s-1}_{2,2}}+||\nabla\delta u||^2_{B^{s-1}_{2,2}}) \dd \tau.
\end{split}\end{align*}
This completes the proof of this lemma.
\end{proof}

Then, in virtue of Lemma \ref{le-estimate}, we will show that the solution to the system \eqref{MHD} depend continuously on the initial data.

\noindent\textbf{Proof of Theorem \ref{th1}.} First, according to Lemma \ref{le0}, there exist a positive $T_n>0$ such that \eqref{MHD} have a solution $(u^n,b^n)\in X_s(T_n)$. Indeed, by Lemma \ref{le0}, we have
$$T_n\geq \frac{1}{2C_s(||u^n_0||^2_{H^{s}}+||b^n_0||^2_{H^{s}})}\geq \frac{1}{2C_sR}:=T,$$
with $R=\sup\limits_{n\geq 0}(||u^n_0||_{H^{s}}^2+||b^n_0||_{H^{s}}^2)$. Moreover, we also obtain from Lemma \ref{le0} that for all $t\in[0,T]$,
\bbal
||u^n(t)||^2_{B^{s}_{2,2}}+||b^n(t)||^2_{B^{s}_{2,2}}+\int^t_0||\na u^n||^2_{B^{s}_{2,2}}\dd \tau
\leq C(||u^n_0||^2_{B^{s}_{2,2}}+||b^n_0||^2_{B^{s}_{2,2}}).
\end{align*}
Since $s+1>1+\frac d2$, we use the similar method as in Lemma \ref{le-estimate} to have for all $t\in[0,T]$,
\bal\label{eq3-0}\begin{split}
&\quad||u^n(t)||^2_{B^{s+1}_{2,2}}+||b^n(t)||^2_{B^{s+1}_{2,2}}+\int^t_0||\na u^n||^2_{B^{s+1}_{2,2}}\dd \tau
\\&\leq (||u^n_0||^2_{B^{s+1}_{2,2}}+||b^n_0||^2_{B^{s+1}_{2,2}})
e^{C\int^t_0(1+||u^n||^2_{B^{s+1}_{2,2}}+||b^n||^2_{B^{s}_{2,2}})\dd \tau}
\\&\leq C(||u^n_0||^2_{B^{s+1}_{2,2}}+||b^n_0||^2_{B^{s+1}_{2,2}}).
\end{split}\end{align}
Let $(u^n_j,b^n_j)\in \mathcal{C}([0,T];B^{s+1}_{2,2})$ be the approximate equations of the system \eqref{MHD}:
\bbal\begin{cases}
\partial_tu^n_j+u^n_j\cdot\nabla u^n_j-\De u^n_j+\nabla P_{n,j}=b^n_j\cdot \nabla b^n_j, \\
\partial_tb^n_j+u^n_j\cdot \nabla b^n_j=b^n_j\cdot\nabla u^n_j,\\
\mathrm{div} u^n_j=\mathrm{div} b^n_j=0,\quad (u^n_j,b^n_j)|_{t=0}=S_j(u^n_0,b^n_0).
\end{cases}\end{align*}
Then, according to Lemma \ref{le-estimate}, we have for all $t\in[0,T]$,
\bal\label{eq3-1}\begin{split}
&\quad||u^n_j(t)-u^n(t)||^2_{B^{s-1}_{2,2}}+||b^n_j(t)-b^n(t)||^2_{B^{s-1}_{2,2}}+\int^t_0||\nabla(u^n_j-u^n)||^2_{B^{s-1}_{2,2}}\dd \tau
\\&\leq  C(||(\mathrm{Id}-S_j)u^n_0||^2_{B^{s-1}_{2,2}}+||(\mathrm{Id}-S_j)b^n_0||^2_{B^{s-1}_{2,2}}).
\end{split}\end{align}
By \eqref{eq3-0}, we have
\bal\label{eq3-2}
||u^n_j||^2_{B^{s+1}_{2,2}}+||b^n_j||^2_{B^{s+1}_{2,2}}\leq C(||S_ju^n_0||^2_{B^{s+1}_{2,2}}+||S_jb^n_0||^2_{B^{s+1}_{2,2}})\leq C2^j.
\end{align}
Using Lemma \ref{le-estimate} again and combining \eqref{eq3-1}-\eqref{eq3-2} yields
\bal\label{eq3-3}\begin{split}
& \quad \ ||u^n_j(t)-u^n(t)||^2_{B^{s}_{2,2}}+||b^n_j(t)-b^n(t)||^2_{B^{s}_{2,2}}\\&\leq C(||(\mathrm{Id}-S_j)u^n_0||^2_{B^{s}_{2,2}}+||(\mathrm{Id}-S_j)b^n_0||^2_{B^{s}_{2,2}}\\& \quad \ +\int^t_0||b^n_j||^2_{B^{s+1}_{2,2}}(||u^n_j-u^n||^2_{B^{s-1}_{2,2}}+||\nabla u^n_j-\nabla u^n||^2_{B^{s-1}_{2,2}}) \dd \tau)
\\&\leq C(||(\mathrm{Id}-S_j)u^n_0||^2_{B^{s}_{2,2}}+||(\mathrm{Id}-S_j)b^n_0||^2_{B^{s}_{2,2}}
\\& \quad \ +2^{2j}||(\mathrm{Id}-S_j)u^n_0||^2_{B^{s-1}_{2,2}}+2^{2j}||(\mathrm{Id}-S_j)b^n_0||^2_{B^{s-1}_{2,2}})
\\& \leq C(||(\mathrm{Id}-S_j)u^n_0||^2_{B^{s}_{2,2}}+||(\mathrm{Id}-S_j)b^n_0||^2_{B^{s}_{2,2}}).
\end{split}\end{align}
It is easy to check that
\bal\label{eq3-4}\begin{split}
&\quad||u^n_j-u^\infty_j||^2_{L^\infty_T(B^{s}_{2,2})}+||b^n_j-b^\infty_j||^2_{L^\infty_T(B^{s}_{2,2})}
\\&\leq ||u^n_j-u^\infty_j||_{L^\infty_T(B^{s-1}_{2,2})}||u^n_j-u^\infty_j||_{L^\infty_T(B^{s+1}_{2,2})}
\\&\quad+||b^n_j-b^\infty_j||_{L^\infty_T(B^{s-1}_{2,2})}||b^n_j-b^\infty_j||_{L^\infty_T(B^{s+1}_{2,2})}
\\&\leq C2^j(||u^n_0-u^\infty_0||_{B^{s}_{2,2}}+||b^n_0-u^\infty_0||_{B^{s}_{2,2}})
\end{split}\end{align}
Therefore, combing \eqref{eq3-3} and \eqref{eq3-4}, we obtain
\bbal
& \quad \ ||u^n-u^\infty||^2_{L^\infty_T(B^{s}_{2,2})}+||b^n-b^\infty||^2_{L^\infty_T(B^{s}_{2,2})}\\&\leq ||u^n_j-u^\infty_j||^2_{L^\infty_T(B^{s}_{2,2})}+||b^n_j-b^\infty_j||^2_{L^\infty_T(B^{s}_{2,2})}
\\&\quad  +||u^n_j-u^n||^2_{L^\infty_T(B^{s}_{2,2})}+||b^n_j-b^n||^2_{L^\infty_T(B^{s}_{2,2})}
\\&\quad +||u^\infty_j-u^\infty||^2_{L^\infty_T(B^{s}_{2,2})}+||b^\infty_j-b^\infty||^2_{L^\infty_T(B^{s}_{2,2})}
\\& \leq  C(||(\mathrm{Id}-S_j)u^n_0||^2_{B^{s}_{2,2}}+||(\mathrm{Id}-S_j)b^n_0||^2_{B^{s}_{2,2}}+||(\mathrm{Id}-S_j)u^\infty_0||^2_{B^{s}_{2,2}}\\& \quad \ +||(\mathrm{Id}-S_j)b^\infty_0||^2_{B^{s}_{2,2}}+
2^{j}||u^n_0-u_0^\infty||^2_{B^{s}_{2,2}}+2^{j}||b^n_0-b_0^\infty||^2_{B^{s}_{2,2}})
\\& \leq  C(||(\mathrm{Id}-S_j)u_0^\infty||^2_{B^{s}_{2,2}}+||(\mathrm{Id}-S_j)b_0^\infty||^2_{B^{s}_{2,2}}+2^{j}||u^n_0-u_0^\infty||_{B^{s}_{2,2}}+2^{j}||b^n_0-b_0^\infty||_{B^{s}_{2,2}}).
\end{align*}
This completes the proof of Theorem \ref{th1}. \\

\section{Non-uniform continuous dependence}

In this section, we will give the proof of the second theorem. Motivated by Himonas and Holliman \cite{HM}, we first construct two sequence approximate solutions where velocity tends to 0 and magnetic is bounded in Sobolev $H^{s}$ norms. Lately, we will show that the distance between approximate solutions and actual solutions is decaying. Finally, we can conclude that Cauchy problem \eqref{MHD} is not uniformly continuous in Sobolev $B^{s}_{2,2}$ spaces.

Set $0<\delta<\frac13$. Now, we choose the magnetic having the following form (see \cite{HM}):
\bbal
b^{\omega,n}=b^{l,\omega,n}(t,x)+b^{h,\omega,n}(t,x), \qquad \omega\in\{\pm1\},\ x\in\R^d,\ t\in \R,
\end{align*}
where $b^{h,\omega,n}$ is the high frequency term
\bbal
b^{h,\omega,n}(t,x)=
\begin{cases}
\mathrm{rot}\phi^h(t,x)=\Big(\pa_2\phi^{h,\omega,n}(t,x),-\pa_1\phi^{h,\omega,n}(t,x)\Big), \qquad d=2,\\
\Big(\pa_2\phi^{h,\omega,n}(t,x),-\pa_1\phi^{h,\omega,n}(t,x),0\Big), \qquad \qquad \qquad \quad d=3,
\end{cases}
\end{align*}
with
\bbal
\phi^{h,\omega,n}(t,x)=
\begin{cases}
n^{-\delta-s-1}\phi(\frac{x_1}{n^\delta})\phi(\frac{x_2}{n^\delta})\sin(n x_2-\omega t), \qquad \quad d=2,\\
n^{-\delta-s-1}\phi(\frac{x_1}{n^\delta})\phi(\frac{x_2}{n^\delta})\sin(n x_2-\omega t)\phi(x_3) \quad d=3,
\end{cases}
\qquad n \in \Z,
\end{align*}
where $\phi\in C_c^\infty(\R)$ satisfies $\text{supp} \ \phi \in [-1,1]$ and $\phi(x)\equiv 1$ for $|x| < \frac{1}{2}$.
Let $(u^{\omega,n},b^{l,\omega,n})$ satisfies the following initial value problem:
\begin{equation}\label{eq-app}\begin{cases}
\partial_tu^{\omega,n}+u^{\omega,n}\cdot\nabla u^{\omega,n}-\Delta u^{\omega,n}+\nabla P^{\omega,n}=b^{l,\omega,n}\cdot \nabla b^{l,\omega,n}, \\
\partial_tb^{l,\omega,n}+u^{\omega,n}\cdot \nabla b^{l,\omega,n}=b^{l,\omega,n}\cdot\nabla u^{\omega,n},\\
\mathrm{div} u^{\omega,n}=\mathrm{div} b^{l,\omega,n}=0,\quad (u^{\omega,n},b^{l,\omega,n})|_{t=0}=(b^{l,\omega,n}_0,b^{l,\omega,n}_0),
\end{cases}\end{equation}
\bbal
b^{l,\omega,n}_0(x)=
\begin{cases}
\mathrm{rot}\phi^l(x)=\Big(\pa_2\phi^{l,\omega,n}(x),-\pa_1\phi^{l,\omega,n}(x)\Big), \qquad d=2,\\
\Big(\pa_2\phi^{l,\omega,n}(x),-\pa_1\phi^{l,\omega,n}(x),0\Big), \qquad \qquad \qquad \quad d=3,
\end{cases}
\end{align*}
with
\bbal
\phi^{l,\omega,n}(x)=
\begin{cases}
-\omega n^{-1+\delta}\Phi_1(\frac{x_1}{n^{\delta}})\Phi_2(\frac{x_2}{n^{\delta}}), \qquad \quad d=2,\\
-\omega n^{-1+\delta}\Phi_1(\frac{x_1}{n^{\delta}})\Phi_2(\frac{x_2}{n^{\delta}})\Phi_2(\frac{x_3}{n^{\delta}}) \quad d=3,
\end{cases}
\qquad n \in \Z,
\end{align*}
where the localizing functions $\Phi_1, \Phi_2 \in C_c^\infty(\R)$ are chosen such that $\Phi'_1= \Phi_2\equiv 1$ on the support of $\phi$. By the well-posedness result, the $(u^{\omega,n},b^{l,\omega,n})$ belong to $C([0,T];B^s_{2,2}\times B^s_{2,2})$ and have lifespan $T\simeq 1$.
Thus, we can find that $(u^{\omega,n},b^{\omega,n})$ satisfies the following equations:
\begin{equation*}\begin{cases}
\partial_tu^{\omega,n}+u^{\omega,n}\cdot\nabla u^{\omega,n}-\Delta u^{\omega,n}+\nabla \widetilde{P}^{\omega,n}=b^{\omega,n}\cdot\na b^{\omega,n}+\D E^{\omega,n}, \\
\partial_tb^{\omega,n}+u^{\omega,n}\cdot \nabla b^{\omega,n}=b^{\omega,n}\cdot\na u^{\omega,n}+F^{\omega,n},\\
\mathrm{div} u^{\omega,n}=\mathrm{div} b^{\omega,n}=0,\quad (u^{\omega,n},b^{\omega,n})|_{t=0}=(b^{l,\omega,n}_0,b^{l,\omega,n}_0+b^{h,\omega,n}_0),
\end{cases}\end{equation*}
where
\bbal
&E^{\omega,n}=-b^{l,\omega,n}\otimes b^{h,\omega,n}-b^{h,\omega,n}\otimes b^{h,\omega,n}-b^{h,\omega,n}\otimes b^{l,\omega,n},
\\&F^{\omega,n}=\underbrace{\pa_tb^{h,\omega,n}+u^{\omega,n}\cd \na b^{h,\omega,n}}_{F^{\omega,n}_1}\underbrace{-b^{h,\omega,n}\cd\na u^{\omega,n}}_{F^{\omega,n}_2}.
\end{align*}
{\bf First, we consider the 2D case.} To estimate the terms $E^{\omega,n}$ and $F^{\omega,n}$ in Sobolev $B^{s-1}_{2,2}$ norm, we first estimate the terms $u^{\omega,n}$ and $b^{\omega,n}$. By \eqref{eq3-0}, Lemma \ref{le0} and Lemma \ref{le3}, we have for any $r\geq 0$ and $t\in[0,T]$,
\bal\label{eq4-0}\begin{split}
&||u^{\omega,n}||_{B^{s+1}_{2,2}}+||b^{l,\omega,n}||_{B^{s+1}_{2,2}}\leq C||b^{l,\omega,n}_0||_{B^{s+1}_{2,2}}\leq Cn^{-1+\delta},
\\&||b^{h,\omega,n}||_{H^r}\leq Cn^{r-s}, \qquad ||b^{h,\omega,n}||_{L^\infty}\leq Cn^{-1}.
\end{split}\end{align}
For the term $E^{\omega,n}$, we also infer from Lemma \ref{le1} that
\bal\label{eq4-0.1}\begin{split}
||E^{\omega,n}||_{B^{s-1}_{2,2}}&\leq C||b^{l,\omega,n}||_{B^{s}_{2,2}}||b^{h,\omega,n}||_{B^{s-1}_{2,2}}+C||b^{h,\omega,n}||_{L^\infty}||b^{h,\omega,n}||_{B^{s-1}_{2,2}}
\\&\leq Cn^{-1}\big(n^{-1+\delta}+n^{-\delta}\big).
\end{split}\end{align}
%and
%\bal\label{eq4-0.2}\begin{split}
%||E^{\omega,n}||_{B^{s}_{2,2}}&\leq C||b^{l,\omega,n}||_{B^{s}_{2,2}}||b^{h,\omega,n}||_{B^{s}_{2,2}}+C||b^{h,\omega,n}||_{L^\infty}||b^{h,\omega,n}||_{B^{s}_{2,2}}
%\\&\leq Cn^{-1+\delta}.
%\end{split}\end{align}
Now, we need to estimate the term $F^{\omega,n}$ . Direct calculation shows that
\bbal
\Big(b^{h,\omega,n}\Big)_1&=n^{-\delta-s}\phi(\frac{x_1}{n^\delta})\phi(\frac{x_2}{n^\delta})\cos(n x_2-\omega t)
\\&\quad+n^{-2\delta-s-1}\phi(\frac{x_1}{n^\delta})\phi'(\frac{x_2}{n^\delta})\sin(n x_2-\omega t),
\end{align*}

\bbal
\Big(b^{h,\omega,n}\Big)_2=-n^{-2\delta-s-1}\phi'(\frac{x_1}{n^\delta})\phi(\frac{x_2}{n^\delta})\sin(n x_2-\omega t),
\end{align*}

\bbal
\Big(\pa_tb^{h,\omega,n}\Big)_1(t,x)&=\omega n^{-s-\delta}\phi(\frac{x_1}{n^\delta})\phi(\frac{x_2}{n^\delta})\sin(nx_2-\omega t)
\\& \quad -\omega n^{-s-1-2\delta}\phi(\frac{x_1}{n^\delta})\phi'(\frac{x_2}{n^\delta})\cos(nx_2-\omega t),
\end{align*}

\bbal
\Big(\pa_tb^{h,\omega,n}\Big)_2(t,x)=\omega n^{-s-1-2\delta}\phi'(\frac{x_1}{n^\delta})\phi(\frac{x_2}{n^\delta})\cos(nx_2-\omega t),
\end{align*}

\bbal
\Big(u^{\omega,n}\cd \na b^{h,\omega,n}\Big)_1(t,x)&=-n^{-s+1-\delta}u^{\omega,n}_2(t,x)\phi(\frac{x_1}{n^\delta})\phi(\frac{x_2}{n^\delta})\sin(nx_2-\omega t)
\\&\quad +n^{-3\delta-s-1}u^{\omega,n}_2(t,x)\phi(\frac{x_1}{n^\delta})\phi''(\frac{x_2}{n^\delta})\sin(n x_2-\omega t)
\\&\quad +2n^{-2\delta-s}u^{\omega,n}_2(t,x)\phi(\frac{x_1}{n^\delta})\phi'(\frac{x_2}{n^\delta})\cos(n x_2-\omega t)
\\&\quad +n^{-2\delta-s}u^{\omega,n}_1(t,x)\phi'(\frac{x_1}{n^\delta})\phi(\frac{x_2}{n^\delta})\cos(n x_2-\omega t)
\\&\quad +n^{-3\delta-s-1}u^{\omega,n}_1(t,x)\phi'(\frac{x_1}{n^\delta})\phi'(\frac{x_2}{n^\delta})\sin(n x_2-\omega t),
\end{align*}

\bbal
\Big(u^{\omega,n}\cd \na b^{h,\omega,n}\Big)_2(t,x)&=-n^{-3\delta-s-1}u^{\omega,n}_1(t,x)\phi''(\frac{x_1}{n^\delta})\phi(\frac{x_2}{n^\delta})\sin(n x_2-\omega t)
\\&\quad -n^{-3\delta-s-1}u^{\omega,n}_2(t,x)\phi'(\frac{x_1}{n^\delta})\phi'(\frac{x_2}{n^\delta})\sin(n x_2-\omega t)
\\&\quad -n^{-2\delta-s}u^{\omega,n}_2(t,x)\phi'(\frac{x_1}{n^\delta})\phi(\frac{x_2}{n^\delta})\cos(n x_2-\omega t).
\end{align*}
Let us handle with the first component of $F^{\omega,n}$. Using the fact
\bbal
&\quad \omega n^{-s-\delta}\phi(\frac{x_1}{n^\delta})\phi(\frac{x_2}{n^\delta})\sin(nx_2-\omega t)
\\&=n^{-s+1-\delta}u^{\omega,n}_2(0,x)\phi(\frac{x_1}{n^\delta})\phi(\frac{x_2}{n^\delta})\sin(nx_2-\omega t),
\end{align*}
we can obtain
\bbal
&\quad\Big(\pa_tb^{h,\omega,n}+u^{\omega,n}\cd \na b^{h,\omega,n}\Big)_1(t,x)
\\&=-n^{-s+1-\delta}[u^{\omega,n}_2(t,x)-u^{\omega,n}_2(0,x)]\phi(\frac{x_1}{n^\delta})\phi(\frac{x_2}{n^\delta})\sin(nx_2-\omega t)
\\& \quad -\omega n^{-s-1-2\delta}\phi(\frac{x_1}{n^\delta})\phi'(\frac{x_2}{n^\delta})\cos(nx_2-\omega t)
\\&\quad +n^{-3\delta-s-1}u^{\omega,n}_2(t,x)\phi(\frac{x_1}{n^\delta})\phi''(\frac{x_2}{n^\delta})\sin(n x_2-\omega t)
\\&\quad +n^{-2\delta-s}u^{\omega,n}_1(t,x)\phi'(\frac{x_1}{n^\delta})\phi(\frac{x_2}{n^\delta})\cos(n x_2-\omega t)
\\&\quad +n^{-3\delta-s-1}u^{\omega,n}_1(t,x)\phi'(\frac{x_1}{n^\delta})\phi'(\frac{x_2}{n^\delta})\sin(n x_2-\omega t)
\\&\quad
+2n^{-2\delta-s}u^{\omega,n}_2(t,x)\phi(\frac{x_1}{n^\delta})\phi'(\frac{x_2}{n^\delta})\cos(n x_2-\omega t):=\sum^6_{i=1}F_{1,1}^{\omega,n}(i).
\end{align*}
To estimathe $F_{1,1}^{\omega,n}(1)$, we can deduce from the first equation of \eqref{eq-app} that
%\bbal
%&\quad\frac12\frac{\dd}{\dd t}||\na u^{\omega,n}||^2_{B^{s}_{2,2}}+||\pa_t u^{\omega,n}||^2_{B^{s}_{2,2}}
%\\&\leq C||-u^{\omega,n}\cd \na u^{\omega,n}+b^{l,\omega,n}\cd\na b^{l,\omega,n}||_{B^{s}_{2,2}}\cdot||\pa_t u^{\omega,n}||_{B^{s}_{2,2}}
%\\&\leq Cn^{2(-1+\delta)}||\pa_t u^{\omega,n}||_{B^{s}_{2,2}},
%\end{align*}
%which implies
%\bal\label{eq4-1}
%||u^{\omega,n}(t,x)-u^{\omega,n}(0,x)||_{B^{s}_{2,2}}&\leq \int^t_0||\pa_\tau u^{\omega,n}(\tau)||_{B^{s}_{2,2}}\dd \tau\leq Cn^{2(-1+\delta)}.
%\end{align}
\bal\label{eq4-1}
&||u^{\omega,n}(t,x)-u^{\omega,n}(0,x)||_{B^{s}_{2,2}} \nonumber\\
&\leq ||u^{\omega,n}(t,x)-e^{t\Delta}u^{\omega,n}(0,x)||_{B^{s}_{2,2}}+||e^{t\Delta}u^{\omega,n}(0,x)-u^{\omega,n}(0,x)||_{B^{s}_{2,2}} \nonumber\\
&\leq ||\int_0^t e^{(t-\tau)\Delta} \mathbf{P} (-u^{\omega,n}\cd \na u^{\omega,n}+b^{l,\omega,n}\cd\na b^{l,\omega,n}) d\tau ||_{B^{s}_{2,2}} \nonumber\\
& \quad + ||\int_0^t \frac{d}{d\tau} (e^{\tau \Delta} u^{\omega,n}(0,x)) d\tau||_{B^{s}_{2,2}} \nonumber\\
&\leq C||-u^{\omega,n}\cd \na u^{\omega,n}+b^{l,\omega,n}\cd\na b^{l,\omega,n}||_{B^{s}_{2,2}} + C||\Delta u^{\omega,n}(0,x)||_{B^{s}_{2,2}} \nonumber\\
&\leq Cn^{2(-1+\delta)}+C n^{-1-\delta},
\end{align}
where $\mathbf{P}:=I+\nabla(-\Delta)^{-1}\operatorname{div}$  is the $d\times d$ matrix pseudo-differential operator in
$\mathbb{R}^d$ with the symbol
$(\delta_{ij}-\frac{\xi_i\xi_j}{|\xi|^2})_{i,j=1}^d$,\ d=2,\ 3.
By Lemma \ref{le1} and \eqref{eq4-1}, we have
\bbal
&||F_{1,1}^{\omega,n}(1)||_{B^{s-1}_{2,2}}\lesssim n^{-2+2\delta}+n^{-1-\delta},
\\&||F_{1,1}^{\omega,n}(2)||_{B^{s-1}_{2,2}}\lesssim n^{-2-\delta},
\\&||F_{1,1}^{\omega,n}(3)||_{B^{s-1}_{2,2}}\lesssim n^{-2-\delta},
\\&||F_{1,1}^{\omega,n}(4)||_{B^{s-1}_{2,2}}\lesssim n^{-1-\delta},
\\&||F_{1,1}^{\omega,n}(5)||_{B^{s-1}_{2,2}}\lesssim n^{-2-2\delta},
\\&||F_{1,1}^{\omega,n}(6)||_{B^{s-1}_{2,2}}\lesssim n^{-1-\delta},
\end{align*}
which yield
\bal\label{eq4-2}
||F^{\omega,n}_{1,1}||_{B^{s-1}_{2,2}}=||\sum^6_{i=1}F_{1,1}^{\omega,n}(i)||_{B^{s-1}_{2,2}}\leq Cn^{-1}\big(n^{-1+2\delta}+n^{-\delta}\big).
\end{align}
For the second component of $F^{\omega,n}$, we also have
\bbal
&\quad\Big(\pa_tb^{h,\omega,n}+u^{\omega,n}\cd \na b^{h,\omega,n}\Big)_2(t,x)\\&=\omega n^{-s-1-2\delta}\phi'(\frac{x_1}{n^\delta})\phi(\frac{x_2}{n^\delta})\cos(nx_2-\omega t)
\\&\quad-n^{-3\delta-s-1}u^{\omega,n}_1(t,x)\phi''(\frac{x_1}{n^\delta})\phi(\frac{x_2}{n^\delta})\sin(n x_2-\omega t)
\\&\quad -n^{-3\delta-s-1}u^{\omega,n}_2(t,x)\phi'(\frac{x_1}{n^\delta})\phi'(\frac{x_2}{n^\delta})\sin(n x_2-\omega t)
\\&\quad -n^{-2\delta-s}u^{\omega,n}_2(t,x)\phi'(\frac{x_1}{n^\delta})\phi(\frac{x_2}{n^\delta})\cos(n x_2-\omega t):=\sum^4_{i=1}F_{1,2}^{\omega,n}(i).
\end{align*}
By Lemma \ref{le1}, we have
\bbal
&||F_{1,2}^{\omega,n}(1)||_{B^{s-1}_{2,2}}\lesssim n^{-2-\delta},
\\&||F_{1,2}^{\omega,n}(2)||_{B^{s-1}_{2,2}}\lesssim n^{-2-2\delta},
\\&||F_{1,2}^{\omega,n}(3)||_{B^{s-1}_{2,2}}\lesssim n^{-2-2\delta},
\\&||F_{1,2}^{\omega,n}(4)||_{B^{s-1}_{2,2}}\lesssim n^{-1-\delta},
\end{align*}
which yield
\bal\label{eq4-3}
||F^{\omega,n}_{1,2}||_{B^{s-1}_{2,2}}=||\sum^4_{i=1}F_{1,2}^{\omega,n}(i)||_{B^{s-1}_{2,2}}\leq Cn^{-1-\delta}.
\end{align}
Therefore, using Lemma \ref{le1}, \eqref{eq4-0} and \eqref{eq4-2}-\eqref{eq4-3}, we obtain
\bal\begin{split}\label{eq4-4}
||F^{\omega,n}||_{B^{s-1}_{2,2}}&\leq ||F^{\omega,n}_{1,1}||_{B^{s-1}_{2,2}}+ ||F^{\omega,n}_{1,2}||_{B^{s-1}_{2,2}}+||b^{h,\omega,n}||_{B^{s-1}_{2,2}}||u^{\omega,n}||_{B^{s+1}_{2,2}}
\\&\leq Cn^{-1}\big(n^{-1+2\delta}+n^{-\delta}\big).
\end{split}\end{align}

Let $(u_{\omega,n},b_{\omega,n})$ be the unique solution of 2D non-resistive MHD equations. Namely,
\begin{equation*}\begin{cases}
\partial_tu_{\omega,n}+u_{\omega,n}\cdot\nabla u_{\omega,n}-\Delta u_{\omega,n}+\nabla P_{\omega,n}=b_{\omega,n}\cdot \nabla b_{\omega,n}, \\
\partial_tb_{\omega,n}+u_{\omega,n}\cdot \nabla b_{\omega,n}=b_{\omega,n}\cdot\nabla u_{\omega,n},\\
\mathrm{div} u_{\omega,n}=\mathrm{div} b_{\omega,n}=0,\quad (u_{\omega,n},b_{\omega,n})|_{t=0}=(b^{l,\omega,n}_0,b^{l,\omega,n}_0+b^{h,\omega,n}(0,x)).
\end{cases}\end{equation*}
By the well-posedness result, the $(u_{\omega,n},b_{\omega,n})$ belong to $C([0,T];B^s_{2,2}\times B^s_{2,2})$ and have common lifespan $T\simeq 1$. Notice that $b^{h,1,n}(0,x)=b^{h,-1,n}(0,x)$, we have
\bal\begin{split}\label{eq4-5}
&\qquad||u_{1,n}(0,x)-u_{-1,n}(0,x)||_{B^{s}_{2,2}}+||b_{1,n}(0,x)-b_{-1,n}(0,x)||_{B^{s}_{2,2}}
\\&\leq C||b^{l,1,n}_0-b^{l,-1,n}_0||_{B^{s}_{2,2}}\leq Cn^{-1+\delta}\rightarrow 0, \qquad n\rightarrow \infty.
\end{split}\end{align}
According to Lemma \ref{le-estimate}, we get from \eqref{eq3-0} that
\bal\begin{split}\label{eq4-6}
&||u_{\omega,n}||^2_{L^\infty_T(B^{s-1}_{2,2})}+||\nabla u_{\omega,n}||^2_{L^2_T(B^{s-1}_{2,2})}+||b_{\omega,n}||^2_{L^\infty_T(B^{s-1}_{2,2})}\leq Cn^{-2+2\delta},
\\&||u_{\omega,n}||^2_{L^\infty_T(B^{s}_{2,2})}+||\nabla u_{\omega,n}||^2_{L^2_T(B^{s}_{2,2})}+||b_{\omega,n}||^2_{L^\infty_T(B^{s}_{2,2})}\leq C,
\\&||u_{\omega,n}||^2_{L^\infty_T(B^{s+1}_{2,2})}+||\nabla u_{\omega,n}||^2_{L^2_T(B^{s+1}_{2,2})}+||b_{\omega,n}||^2_{L^\infty_T(B^{s+1}_{2,2})}\leq Cn^2
\end{split}\end{align}
Next, considering the difference $v=u_{\omega,n}-u^{\omega,n}$ and $a=b_{\omega,n}-b^{\omega,n}$, we observe that $(v,a)$ satifies
\begin{equation*}\begin{cases}
\partial_tv-\Delta v+\nabla \widetilde{P}_{\omega,n}=\D E^{\omega,n}+\D G, \\
\partial_ta+u_{\omega,n}\cdot \nabla a=F^{\omega,n}+J,\\
\mathrm{div} v=\mathrm{div} a=0,\quad (v,a)|_{t=0}=(0,0),
\end{cases}\end{equation*}
where
\bbal
&G=\underbrace{b_{\omega,n}\otimes a+a\otimes b^{\omega,n}}_{G_1}\underbrace{-v\otimes u^{\omega,n}-u_{\omega,n}\otimes v}_{G_2}, \
\\& J=\underbrace{b_{\omega,n}\cdot\nabla v}_{J_1}+\underbrace{a\cdot\nabla u^{\omega,n}}_{J_2}\underbrace{-v\cd\na b^{\omega,n}}_{J_3}.
\end{align*}
Next, we shall bound the term $(v,a)$ in Sobolev $B^{s-1}_{2,2}$ norm. According to Lemma \ref{le-estimate}, we have
\bal\begin{split}\label{eq4-7}
&\quad\frac{\dd}{\dd t}\big(||v||^2_{B^{s-1}_{2,2}}+||a||^2_{B^{s-1}_{2,2}}\big)+||\nabla v||^2_{B^{s-1}_{2,2}}
\\&\leq C||E^{\omega,n}||^2_{B^{s-1}_{2,2}}+C||G||^2_{B^{s-1}_{2,2}}+C||F^{\omega,n}||_{B^{s-1}_{2,2}}||a||_{B^{s-1}_{2,2}}
\\&\quad+C||J||_{B^{s-1}_{2,2}}||a||_{B^{s-1}_{2,2}}+\frac14||\nabla v||^2_{B^{s-1}_{2,2}}.
\end{split}\end{align}
By Lemma \ref{le1}, \eqref{eq4-0} and \eqref{eq4-6}, we get
\bal\label{eq4-8}\begin{split}
&||G_1||_{B^{s-1}_{2,2}}\leq C||a||_{B^{s-1}_{2,2}}\big(||b^{\omega,n}||_{B^{s}_{2,2}}+||b_{\omega,n}||_{B^{s}_{2,2}}\big),
\\&||G_2||_{B^{s-1}_{2,2}}\leq C||a||_{B^{s-1}_{2,2}}\big(||u^{\omega,n}||_{B^{s}_{2,2}}+||u_{\omega,n}||_{B^{s}_{2,2}}\big),
\\&||J_1||_{B^{s-1}_{2,2}}\leq C||b_{\omega,n}||_{B^{s}_{2,2}}||\nabla v||_{B^{s-1}_{2,2}},
\\&||J_2||_{B^{s-1}_{2,2}}\leq C||a||_{B^{s-1}_{2,2}}||u^{\omega,n}||_{B^{s+1}_{2,2}},
\\&||J_3||_{B^{s-1}_{2,2}}\leq C||v||_{B^{s}_{2,2}}||b^{\omega,n}||_{B^{s}_{2,2}}\leq C||v,\nabla v||_{B^{s-1}_{2,2}}||b^{\omega,n}||_{B^{s}_{2,2}}.
\end{split}\end{align}
Plugging \eqref{eq4-8} into \eqref{eq4-7} and using \eqref{eq4-0.1}, \eqref{eq4-4}, we have for all $t\in[0,T]$,
\bal\label{eq4-9}\begin{split}
&\quad\frac{\dd}{\dd t}\big(||v||^2_{B^{s-1}_{2,2}}+||a||^2_{B^{s-1}_{2,2}}\big)
\\&\leq C\big(||v||^2_{B^{s-1}_{2,2}}+||a||^2_{B^{s-1}_{2,2}}\big)+||E^{\omega,n}||^2_{B^{s-1}_{2,2}}+||F^{\omega,n}||^2_{B^{s-1}_{2,2}}
\\&\leq C\big(||v||^2_{B^{s-1}_{2,2}}+||a||^2_{B^{s-1}_{2,2}}\big)+Cn^{-2}\big(n^{-1+2\delta}+n^{-\delta}\big)^2.
\end{split}\end{align}
This alongs with Growall's inequallity yields
\bal\label{eq4-9-1}
||v||^2_{L^\infty_T(B^{s-1}_{2,2})}+||a||^2_{L^\infty_T(B^{s-1}_{2,2})}\leq Cn^{-2}\big(n^{-1+2\delta}+n^{-\delta}\big)^2.
\end{align}
According to the interpolation inequality and using \eqref{eq4-0}, \eqref{eq4-6}, \eqref{eq4-9-1}, we have
\bal\label{eq4-10}\begin{split}
&\quad||v||^2_{{L^\infty_T(B^{s}_{2,2})}}+||a||^2_{{L^\infty_T(B^{s}_{2,2})}}
\\&\leq \big(||v||^2_{L^\infty_T(B^{s-1}_{2,2})}+||a||^2_{L^\infty_T(B^{s-1}_{2,2})}\big)^{\frac12}\big(||v||^2_{L^\infty_T(B^{s+1}_{2,2})}+||a||^2_{L^\infty_T(B^{s+1}_{2,2})}\big)^{\frac12}
\\&\leq C\big(n^{-1+2\delta}+n^{-\delta}\big).
\end{split}\end{align}
Then, combining \eqref{eq4-0} and \eqref{eq4-10}, we have
\begin{align*}\begin{split}
&\quad||u_{1,n}(t)-u_{-1,n}(t)||_{B^{s}_{2,2}}+||b_{1,n}(t)-b_{-1,n}(t)||_{B^{s}_{2,2}}
\\&\geq ||u^{1,n}(t)-u^{-1,n}(t)||_{B^{s}_{2,2}}+||b^{1,n}(t)-b^{-1,n}(t)||_{B^{s}_{2,2}}-C\varepsilon_n
\\&\geq ||b^{h,1,n}(t)-b^{h,-1,n}(t)||_{B^{s}_{2,2}}-C\varepsilon'_n
\\&\geq 2|\sin t|\cdot||n^{-\delta-s}\phi(\frac{x_1}{n^\delta})\phi(\frac{x_2}{n^\delta})\sin(n x_2)||_{B^{s}_{2,2}}-C\varepsilon'_n
\\&\geq c_0|\sin t|\cdot||n^{-\frac12\delta}\phi(\frac{x_1}{n^\delta})||_{H^{s}(\R)}||n^{-\frac12\delta-s}\phi(\frac{x_2}{n^\delta})\sin(n x_2)||_{H^{s}(\R)}-C\varepsilon'_n,
\end{split}\end{align*}
where
\bbal
\varepsilon_n=(n^{-\delta}+n^{2\delta-1})^{\frac12}, \qquad \varepsilon'_n=(n^{-\delta}+n^{2\delta-1})^{\frac12}+n^{\delta-1}.
\end{align*}
Letting $n$ go to $\infty$, then there exists some positive constant $c_0$ such that
\begin{align}\label{eq4-11}\begin{split}
&\liminf_{n\rightarrow \infty}\big(||u_{1,n}(t)-u_{-1,n}(t)||_{B^{s}_{2,2}}+||b_{1,n}(t)-b_{-1,n}(t)||_{B^{s}_{2,2}}\big)
\\&\geq c_0|\sin t|.
\end{split}\end{align}
Then \eqref{eq4-11} together with \eqref{eq4-5} complete the proof of the 2D case of Theorem \ref{th2}.

{\bf Now, we consider the 3D case.} Similarly, to estimate the terms $E^{\omega,n}$ and $F^{\omega,n}$ in Sobolev $B^{s-1}_{2,2}$ norm, we first estimate the terms $u^{\omega,n}$ and $b^{\omega,n}$. By \eqref{eq3-0}, Lemma \ref{le0} and Lemma \ref{le3}, we get for any $r\geq 0$ and $t\in[0,T]$,
\bal\label{eq5-0}\begin{split}
&||u^{\omega,n}||_{B^{s+1}_{2,2}}+||b^{l,\omega,n}||_{B^{s+1}_{2,2}}\leq C||b^{l,\omega,n}_0||_{B^{s+1}_{2,2}}\leq Cn^{-1+\frac32\delta},
\\&||b^{h,\omega,n}||_{H^r}\leq Cn^{r-s}, \qquad ||b^{h,\omega,n}||_{L^\infty}\leq Cn^{-1}.
\end{split}\end{align}
For $E^{\omega,n}$, by Lemma \ref{le1}, we have
\bal\label{eq5-0.1}\begin{split}
||E^{\omega,n}||_{B^{s-1}_{2,2}}&\leq C||b^{l,\omega,n}||_{B^{s}_{2,2}}||b^{h,\omega,n}||_{B^{s-1}_{2,2}}+C||b^{h,\omega,n}||_{L^\infty}||b^{h,\omega,n}||_{B^{s-1}_{2,2}}
\\&\leq Cn^{-1}\big(n^{-1+\frac32\delta}+n^{-\delta}\big).
\end{split}\end{align}
%and
%\bal\label{eq5-0.2}\begin{split}
%||E^{\omega,n}||_{B^{s}_{2,2}}&\leq C||b^{l,\omega,n}||_{B^{s}_{2,2}}||b^{h,\omega,n}||_{B^{s}_{2,2}}+C||b^{h,\omega,n}||_{L^\infty}||b^{h,\omega,n}||_{B^{s}_{2,2}}
%\\&\leq Cn^{-1+\delta}.
%\end{split}\end{align}
Now, we estimate $F^{\omega,n}$ . Direct calculation shows that
\bbal
\Big(\pa_tb^{h,\omega,n}+u^{\omega,n}\cd \na b^{h,\omega,n}\Big)_3(t,x)=0,
\end{align*}

\bbal
\Big(b^{h,\omega,n}\Big)_1&=n^{-\delta-s}\phi(\frac{x_1}{n^\delta})\phi(\frac{x_2}{n^\delta})\cos(n x_2-\omega t)\phi(x_3)
\\&\quad+n^{-2\delta-s-1}\phi(\frac{x_1}{n^\delta})\phi'(\frac{x_2}{n^\delta})\sin(n x_2-\omega t)\phi(x_3),
\end{align*}

\bbal
\Big(b^{h,\omega,n}\Big)_2=-n^{-2\delta-s-1}\phi'(\frac{x_1}{n^\delta})\phi(\frac{x_2}{n^\delta})\sin(n x_2-\omega t)\phi(x_3),
\end{align*}

\bbal
\Big(\pa_tb^{h,\omega,n}\Big)_1(t,x)&=\omega n^{-s-\delta}\phi(\frac{x_1}{n^\delta})\phi(\frac{x_2}{n^\delta})\sin(nx_2-\omega t)\phi(x_3)
\\& \quad -\omega n^{-s-1-2\delta}\phi(\frac{x_1}{n^\delta})\phi'(\frac{x_2}{n^\delta})\cos(nx_2-\omega t)\phi(x_3),
\end{align*}

\bbal
\Big(\pa_tb^{h,\omega,n}\Big)_2(t,x)=\omega n^{-s-1-2\delta}\phi'(\frac{x_1}{n^\delta})\phi(\frac{x_2}{n^\delta})\cos(nx_2-\omega t)\phi(x_3),
\end{align*}

\bbal
\Big(u^{\omega,n}\cd \na b^{h,\omega,n}\Big)_1(t,x)&=-n^{-s+1-\delta}u^{\omega,n}_2(t,x)\phi(\frac{x_1}{n^\delta})\phi(\frac{x_2}{n^\delta})\sin(nx_2-\omega t)\phi(x_3)
\\&\quad +n^{-3\delta-s-1}u^{\omega,n}_2(t,x)\phi(\frac{x_1}{n^\delta})\phi''(\frac{x_2}{n^\delta})\sin(n x_2-\omega t)\phi(x_3)
\\&\quad +2n^{-2\delta-s}u^{\omega,n}_2(t,x)\phi(\frac{x_1}{n^\delta})\phi'(\frac{x_2}{n^\delta})\cos(n x_2-\omega t)\phi(x_3)
\\&\quad +n^{-2\delta-s}u^{\omega,n}_1(t,x)\phi'(\frac{x_1}{n^\delta})\phi(\frac{x_2}{n^\delta})\cos(n x_2-\omega t)\phi(x_3)
\\&\quad +n^{-3\delta-s-1}u^{\omega,n}_1(t,x)\phi'(\frac{x_1}{n^\delta})\phi'(\frac{x_2}{n^\delta})\sin(n x_2-\omega t)\phi(x_3)
\\&\quad +n^{-\delta-s}u^{\omega,n}_3(t,x)\phi(\frac{x_1}{n^\delta})\phi(\frac{x_2}{n^\delta})\cos(n x_2-\omega t)\pa_3\phi(x_3)
\\&\quad +n^{-2\delta-s-1}u^{\omega,n}_3(t,x)\phi(\frac{x_1}{n^\delta})\phi'(\frac{x_2}{n^\delta})\sin(n x_2-\omega t)\pa_3\phi(x_3),
\end{align*}

\bbal
\Big(u^{\omega,n}\cd \na b^{h,\omega,n}\Big)_2(t,x)&=-n^{-3\delta-s-1}u^{\omega,n}_1(t,x)\phi''(\frac{x_1}{n^\delta})\phi(\frac{x_2}{n^\delta})\sin(n x_2-\omega t)\phi(x_3)
\\&\quad -n^{-3\delta-s-1}u^{\omega,n}_2(t,x)\phi'(\frac{x_1}{n^\delta})\phi'(\frac{x_2}{n^\delta})\sin(n x_2-\omega t)\phi(x_3)
\\&\quad -n^{-2\delta-s}u^{\omega,n}_2(t,x)\phi'(\frac{x_1}{n^\delta})\phi(\frac{x_2}{n^\delta})\cos(n x_2-\omega t)\phi(x_3)
\\&\quad -n^{-2\delta-s-1}u^{\omega,n}_3(t,x)\phi'(\frac{x_1}{n^\delta})\phi(\frac{x_2}{n^\delta})\sin(n x_2-\omega t)\pa_3\phi(x_3).
\end{align*}
For the first component of $F^{\omega,n}$. It is easy to see that
\bbal
&\quad \omega n^{-s-\delta}\phi(\frac{x_1}{n^\delta})\phi(\frac{x_2}{n^\delta})\sin(nx_2-\omega t)\phi(x_3)
\\&=n^{-s+1-\delta}u^{\omega,n}_2(0,x)\phi(\frac{x_1}{n^\delta})\phi(\frac{x_2}{n^\delta})\sin(nx_2-\omega t)\phi(x_3),
\end{align*}
we can obtain
\bbal
&\quad\Big(\pa_tb^h+u^{\omega,n}\cd \na b^{h,\omega,n}\Big)_1(t,x)
\\&=-n^{-s+1-\delta}[u^{\omega,n}_2(t,x)-u^{\omega,n}_2(0,x)]\phi(\frac{x_1}{n^\delta})\phi(\frac{x_2}{n^\delta})\sin(nx_2-\omega t)\phi(x_3)
\\& \quad -\omega n^{-s-1-2\delta}\phi(\frac{x_1}{n^\delta})\phi'(\frac{x_2}{n^\delta})\cos(nx_2-\omega t)\phi(x_3)
\\&\quad +n^{-3\delta-s-1}u^{\omega,n}_2(t,x)\phi(\frac{x_1}{n^\delta})\phi''(\frac{x_2}{n^\delta})\sin(n x_2-\omega t)\phi(x_3)
\\&\quad +n^{-2\delta-s}u^{\omega,n}_1(t,x)\phi'(\frac{x_1}{n^\delta})\phi(\frac{x_2}{n^\delta})\cos(n x_2-\omega t)\phi(x_3)
\\&\quad +n^{-3\delta-s-1}u^{\omega,n}_1(t,x)\phi'(\frac{x_1}{n^\delta})\phi'(\frac{x_2}{n^\delta})\sin(n x_2-\omega t)\phi(x_3)
\\&\quad
+2n^{-2\delta-s}u^{\omega,n}_2(t,x)\phi(\frac{x_1}{n^\delta})\phi'(\frac{x_2}{n^\delta})\cos(n x_2-\omega t)\phi(x_3)
\\&\quad +n^{-\delta-s}u^{\omega,n}_3(t,x)\phi(\frac{x_1}{n^\delta})\phi(\frac{x_2}{n^\delta})\cos(n x_2-\omega t)\pa_3\phi(x_3)
\\&\quad +n^{-2\delta-s-1}u^{\omega,n}_3(t,x)\phi(\frac{x_1}{n^\delta})\phi'(\frac{x_2}{n^\delta})\sin(n x_2-\omega t)\pa_3\phi(x_3):=\sum^8_{i=1}F_{1,1}^{\omega,n}(i).
\end{align*}
To estimathe $F_{1,1}^{\omega,n}(1)$, similarly to the 2D case, we get
%\bbal
%&\quad\frac12\frac{\dd}{\dd t}||\na u^{\omega,n}||^2_{B^{s}_{2,2}}+||\pa_t u^{\omega,n}||^2_{B^{s}_{2,2}}
%\\&\leq C||-u^{\omega,n}\cd \na u^{\omega,n}+b^{l,\omega,n}\cd\na b^{l,\omega,n}||_{B^{s}_{2,2}}\cdot||\pa_t u^{\omega,n}||_{B^{s}_{2,2}}
%\\&\leq Cn^{2(-1+\delta)}||\pa_t u^{\omega,n}||_{B^{s}_{2,2}},
%\end{align*}
%which implies
%\bal\label{eq5-1}
%||u^{\omega,n}(t,x)-u^{\omega,n}(0,x)||_{B^{s}_{2,2}}&\leq \int^t_0||\pa_\tau u^{\omega,n}(\tau)||_{B^{s}_{2,2}}\dd \tau\leq Cn^{2(-1+\delta)}.
%\end{align}
\bal\label{eq5-1}
&||u^{\omega,n}(t,x)-u^{\omega,n}(0,x)||_{B^{s}_{2,2}} \nonumber\\
&\leq ||u^{\omega,n}(t,x)-e^{t\Delta}u^{\omega,n}(0,x)||_{B^{s}_{2,2}}+||e^{t\Delta}u^{\omega,n}(0,x)-u^{\omega,n}(0,x)||_{B^{s}_{2,2}} \nonumber\\
&\leq ||\int_0^t e^{(t-\tau)\Delta} \mathbf{P} (-u^{\omega,n}\cd \na u^{\omega,n}+b^{l,\omega,n}\cd\na b^{l,\omega,n}) d\tau ||_{B^{s}_{2,2}} \nonumber\\
& \quad + ||\int_0^t \frac{d}{d\tau} (e^{\tau \Delta} u^{\omega,n}(0,x)) d\tau||_{B^{s}_{2,2}} \nonumber\\
&\leq C||-u^{\omega,n}\cd \na u^{\omega,n}+b^{l,\omega,n}\cd\na b^{l,\omega,n}||_{B^{s}_{2,2}} + C||\Delta u^{\omega,n}(0,x)||_{B^{s}_{2,2}} \nonumber\\
&\leq Cn^{2(-1+\frac32\delta)}+C n^{-1-\frac12\delta}.
\end{align}
By Lemma \ref{le1} and \eqref{eq5-1}, we have
\bbal
&||F_{1,1}^{\omega,n}(1)||_{B^{s-1}_{2,2}}\lesssim n^{-2+3\delta}+n^{-1-\frac12\delta},
\\&||F_{1,1}^{\omega,n}(2)||_{B^{s-1}_{2,2}}\lesssim n^{-2-\delta},
\\&||F_{1,1}^{\omega,n}(3)||_{B^{s-1}_{2,2}}\lesssim n^{-2-\delta},
\\&||F_{1,1}^{\omega,n}(4)||_{B^{s-1}_{2,2}}\lesssim n^{-1-\delta},
\\&||F_{1,1}^{\omega,n}(5)||_{B^{s-1}_{2,2}}\lesssim n^{-2-2\delta},
\\&||F_{1,1}^{\omega,n}(6)||_{B^{s-1}_{2,2}}\lesssim n^{-1-\delta},
\\&||F_{1,1}^{\omega,n}(7)||_{B^{s-1}_{2,2}}\lesssim n^{-2+\frac32\delta},
\\&||F_{1,1}^{\omega,n}(8)||_{B^{s-1}_{2,2}}\lesssim n^{-2-\delta},
\end{align*}
which yield
\bal\label{eq5-2}
||F^{\omega,n}_{1,1}||_{B^{s-1}_{2,2}}=||\sum^8_{i=1}F_{1,1}^{\omega,n}(i)||_{B^{s-1}_{2,2}}\leq Cn^{-1}\big(n^{-1+3\delta}+n^{-\frac12\delta}\big).
\end{align}
For the second component of $F^{\omega,n}$, we also have
\bbal
&\quad\Big(\pa_tb^h+u^{\omega,n}\cd \na b^{h,\omega,n}\Big)_2(t,x)\\&=\omega n^{-s-1-2\delta}\phi'(\frac{x_1}{n^\delta})\phi(\frac{x_2}{n^\delta})\cos(nx_2-\omega t)\phi(x_3)
\\&\quad-n^{-3\delta-s-1}u^{\omega,n}_1(t,x)\phi''(\frac{x_1}{n^\delta})\phi(\frac{x_2}{n^\delta})\sin(n x_2-\omega t)\phi(x_3)
\\&\quad -n^{-3\delta-s-1}u^{\omega,n}_2(t,x)\phi'(\frac{x_1}{n^\delta})\phi'(\frac{x_2}{n^\delta})\sin(n x_2-\omega t)\phi(x_3)
\\&\quad -n^{-2\delta-s}u^{\omega,n}_2(t,x)\phi'(\frac{x_1}{n^\delta})\phi(\frac{x_2}{n^\delta})\cos(n x_2-\omega t)\phi(x_3)
\\&\quad -n^{-2\delta-s-1}u^{\omega,n}_3(t,x)\phi'(\frac{x_1}{n^\delta})\phi(\frac{x_2}{n^\delta})\sin(n x_2-\omega t)\pa_3\phi(x_3):=\sum^5_{i=1}F_{1,2}^{\omega,n}(i).
\end{align*}
By Lemma \ref{le1}, we obtain
\bbal
&||F_{1,2}^{\omega,n}(1)||_{B^{s-1}_{2,2}}\lesssim n^{-2-\delta},
\\&||F_{1,2}^{\omega,n}(2)||_{B^{s-1}_{2,2}}\lesssim n^{-2-2\delta},
\\&||F_{1,2}^{\omega,n}(3)||_{B^{s-1}_{2,2}}\lesssim n^{-2-2\delta},
\\&||F_{1,2}^{\omega,n}(4)||_{B^{s-1}_{2,2}}\lesssim n^{-1-\delta},
\\&||F_{1,2}^{\omega,n}(5)||_{B^{s-1}_{2,2}}\lesssim n^{-2-\delta},
\end{align*}
which yield
\bal\label{eq5-3}
||F^{\omega,n}_{1,2}||_{B^{s-1}_{2,2}}=||\sum^4_{i=1}F_{1,2}^{\omega,n}(i)||_{B^{s-1}_{2,2}}\leq Cn^{-1-\delta}.
\end{align}
Thus, applying Lemma \ref{le1}, \eqref{eq5-0} and \eqref{eq5-2}-\eqref{eq5-3}, we obtain
\bal\begin{split}\label{eq5-4}
||F^{\omega,n}||_{B^{s-1}_{2,2}}&\leq ||F^{\omega,n}_{1,1}||_{B^{s-1}_{2,2}}+ ||F^{\omega,n}_{1,2}||_{B^{s-1}_{2,2}}+||b^{h,\omega,n}||_{B^{s-1}_{2,2}}||u^{\omega,n}||_{B^{s+1}_{2,2}}
\\&\leq Cn^{-1}\big(n^{-1+3\delta}+n^{-\frac12\delta}\big).
\end{split}\end{align}

Let $(u_{\omega,n},b_{\omega,n})$ be the unique solution of 3D non-resistive MHD equations. Namely,
\begin{equation*}\begin{cases}
\partial_tu_{\omega,n}+u_{\omega,n}\cdot\nabla u_{\omega,n}-\Delta u_{\omega,n}+\nabla P_{\omega,n}=b_{\omega,n}\cdot \nabla b_{\omega,n}, \\
\partial_tb_{\omega,n}+u_{\omega,n}\cdot \nabla b_{\omega,n}=b_{\omega,n}\cdot\nabla u_{\omega,n},\\
\mathrm{div} u_{\omega,n}=\mathrm{div} b_{\omega,n}=0,\quad (u_{\omega,n},b_{\omega,n})|_{t=0}=(b^{l,\omega,n}_0,b^{l,\omega,n}_0+b^{h,\omega,n}(0,x)).
\end{cases}\end{equation*}
By the well-posedness result, the $(u_{\omega,n},b_{\omega,n})$ belong to $C([0,T];B^s_{2,2}\times B^s_{2,2})$ and have common lifespan $T\simeq 1$. Notice that $b^{h,1,n}(0,x)=b^{h,-1,n}(0,x)$, we obtain
\bal\begin{split}\label{eq5-5}
&\qquad||u_{1,n}(0,x)-u_{-1,n}(0,x)||_{B^{s}_{2,2}}+||b_{1,n}(0,x)-b_{-1,n}(0,x)||_{B^{s}_{2,2}}
\\&\leq C||b^{l,1,n}_0-b^{l,-1,n}_0||_{B^{s}_{2,2}}\leq Cn^{-1+\frac32\delta}\rightarrow 0, \qquad n\rightarrow \infty.
\end{split}\end{align}
By Lemma \ref{le-estimate}, we get from \eqref{eq3-0} that
\bal\begin{split}\label{eq5-6}
&||u_{\omega,n}||^2_{L^\infty_T(B^{s-1}_{2,2})}+||\nabla u_{\omega,n}||^2_{L^2_T(B^{s-1}_{2,2})}+||b_{\omega,n}||^2_{L^\infty_T(B^{s-1}_{2,2})}\leq Cn^{-2+3\delta},
\\&||u_{\omega,n}||^2_{L^\infty_T(B^{s}_{2,2})}+||\nabla u_{\omega,n}||^2_{L^2_T(B^{s}_{2,2})}+||b_{\omega,n}||^2_{L^\infty_T(B^{s}_{2,2})}\leq C,
\\&||u_{\omega,n}||^2_{L^\infty_T(B^{s+1}_{2,2})}+||\nabla u_{\omega,n}||^2_{L^2_T(B^{s+1}_{2,2})}+||b_{\omega,n}||^2_{L^\infty_T(B^{s+1}_{2,2})}\leq Cn^2
\end{split}\end{align}
Next, considering $v=u_{\omega,n}-u^{\omega,n}$ and $a=b_{\omega,n}-b^{\omega,n}$, we observe that $(v,a)$ satifies
\begin{equation*}\begin{cases}
\partial_tv-\Delta v+\nabla \widetilde{P}_{\omega,n}=\D E^{\omega,n}+\D G, \\
\partial_ta+u_{\omega,n}\cdot \nabla a=F^{\omega,n}+J,\\
\mathrm{div} v=\mathrm{div} a=0,\quad (v,a)|_{t=0}=(0,0),
\end{cases}\end{equation*}
where
\bbal
&G=\underbrace{b_{\omega,n}\otimes a+a\otimes b^{\omega,n}}_{G_1}\underbrace{-v\otimes u^{\omega,n}-u_{\omega,n}\otimes v}_{G_2}, \
\\& J=\underbrace{b_{\omega,n}\cdot\nabla v}_{J_1}+\underbrace{a\cdot\nabla u^{\omega,n}}_{J_2}\underbrace{-v\cd\na b^{\omega,n}}_{J_3}.
\end{align*}
Next, we shall bound the term $(v,a)$ in Sobolev $B^{s-1}_{2,2}$ norm. By Lemma \ref{le-estimate}, we obtain
\bal\begin{split}\label{eq5-7}
&\quad\frac{\dd}{\dd t}\big(||v||^2_{B^{s-1}_{2,2}}+||a||^2_{B^{s-1}_{2,2}}\big)+||\nabla v||^2_{B^{s-1}_{2,2}}
\\&\leq C||E^{\omega,n}||^2_{B^{s-1}_{2,2}}+C||G||^2_{B^{s-1}_{2,2}}+C||F^{\omega,n}||_{B^{s-1}_{2,2}}||a||_{B^{s-1}_{2,2}}
\\&\quad+C||J||_{B^{s-1}_{2,2}}||a||_{B^{s-1}_{2,2}}+\frac14||\nabla v||^2_{B^{s-1}_{2,2}}.
\end{split}\end{align}
By Lemma \ref{le1}, \eqref{eq5-0} and \eqref{eq5-6}, we get
\bal\label{eq5-8}\begin{split}
&||G_1||_{B^{s-1}_{2,2}}\leq C||a||_{B^{s-1}_{2,2}}\big(||b^{\omega,n}||_{B^{s}_{2,2}}+||b_{\omega,n}||_{B^{s}_{2,2}}\big),
\\&||G_2||_{B^{s-1}_{2,2}}\leq C||a||_{B^{s-1}_{2,2}}\big(||u^{\omega,n}||_{B^{s}_{2,2}}+||u_{\omega,n}||_{B^{s}_{2,2}}\big),
\\&||J_1||_{B^{s-1}_{2,2}}\leq C||b_{\omega,n}||_{B^{s}_{2,2}}||\nabla v||_{B^{s-1}_{2,2}},
\\&||J_2||_{B^{s-1}_{2,2}}\leq C||a||_{B^{s-1}_{2,2}}||u^{\omega,n}||_{B^{s+1}_{2,2}},
\\&||J_3||_{B^{s-1}_{2,2}}\leq C||v||_{B^{s}_{2,2}}||b^{\omega,n}||_{B^{s}_{2,2}}\leq C||v,\nabla v||_{B^{s-1}_{2,2}}||b^{\omega,n}||_{B^{s}_{2,2}}.
\end{split}\end{align}
Plugging \eqref{eq5-8} into \eqref{eq5-7} and applying \eqref{eq5-0.1}, \eqref{eq5-4}, we obtain for all $t\in[0,T]$,
\bal\label{eq5-9}\begin{split}
&\quad\frac{\dd}{\dd t}\big(||v||^2_{B^{s-1}_{2,2}}+||a||^2_{B^{s-1}_{2,2}}\big)
\\&\leq C\big(||v||^2_{B^{s-1}_{2,2}}+||a||^2_{B^{s-1}_{2,2}}\big)+||E^{\omega,n}||^2_{B^{s-1}_{2,2}}+||F^{\omega,n}||^2_{B^{s-1}_{2,2}}
\\&\leq C\big(||v||^2_{B^{s-1}_{2,2}}+||a||^2_{B^{s-1}_{2,2}}\big)+Cn^{-2}\big(n^{-1+3\delta}+n^{-\frac12\delta}\big)^2.
\end{split}\end{align}
This alongs with Growall's inequallity yields
\bal\label{eq5-9-1}
||v||^2_{L^\infty_T(B^{s-1}_{2,2})}+||a||^2_{L^\infty_T(B^{s-1}_{2,2})}\leq Cn^{-2}\big(n^{-1+3\delta}+n^{-\frac12\delta}\big)^2.
\end{align}
By \eqref{eq5-0}, \eqref{eq5-6}, \eqref{eq5-9-1} and using the interpolation inequality, we have
\bal\label{eq5-10}\begin{split}
&\quad||v||^2_{{L^\infty_T(B^{s}_{2,2})}}+||a||^2_{{L^\infty_T(B^{s}_{2,2})}}
\\&\leq \big(||v||^2_{L^\infty_T(B^{s-1}_{2,2})}+||a||^2_{L^\infty_T(B^{s-1}_{2,2})}\big)^{\frac12}\big(||v||^2_{L^\infty_T(B^{s+1}_{2,2})}+||a||^2_{L^\infty_T(B^{s+1}_{2,2})}\big)^{\frac12}
\\&\leq C\big(n^{-1+3\delta}+n^{-\frac12\delta}\big).
\end{split}\end{align}
Then, combining \eqref{eq5-0} and \eqref{eq5-10}, we get
\begin{align*}\begin{split}
&\quad||u_{1,n}(t)-u_{-1,n}(t)||_{B^{s}_{2,2}}+||b_{1,n}(t)-b_{-1,n}(t)||_{B^{s}_{2,2}}
\\&\geq ||u^{1,n}(t)-u^{-1,n}(t)||_{B^{s}_{2,2}}+||b^{1,n}(t)-b^{-1,n}(t)||_{B^{s}_{2,2}}-C\varepsilon_n
\\&\geq ||b^{h,1,n}(t)-b^{h,-1,n}(t)||_{B^{s}_{2,2}}-C\varepsilon'_n
\\&\geq 2|\sin t|\cdot||n^{-\delta-s}\phi(\frac{x_1}{n^\delta})\phi(\frac{x_2}{n^\delta})\sin(n x_2)\phi(x_3)||_{B^{s}_{2,2}}-C\varepsilon'_n
\\&\geq c_0|\sin t|\cdot||n^{-\frac12\delta}\phi(\frac{x_1}{n^\delta})||_{H^{s}(\R)}||n^{-\frac12\delta-s}\phi(\frac{x_2}{n^\delta})\sin(n x_2)||_{H^{s}(\R)}||\phi(x_3)||_{H^s(\R)}-C\varepsilon'_n,
\end{split}\end{align*}
where
\bbal
\varepsilon_n=(n^{-\frac12\delta}+n^{3\delta-1})^{\frac12}, \qquad \varepsilon'_n=(n^{-\frac12\delta}+n^{3\delta-1})^{\frac12}+n^{\frac32\delta-1}.
\end{align*}
Letting $n$ go to $\infty$, then there exists some positive constant $\widetilde{c_0}$ such that
\begin{align}\label{eq5-11}\begin{split}
&\liminf_{n\rightarrow \infty}\big(||u_{1,n}(t)-u_{-1,n}(t)||_{B^{s}_{2,2}}+||b_{1,n}(t)-b_{-1,n}(t)||_{B^{s}_{2,2}}\big)
\\&\geq \widetilde{c_0}|\sin t|.
\end{split}\end{align}
Then \eqref{eq5-11} together with \eqref{eq5-5} complete the proof of 3D case.

\vspace*{1em}
\noindent\textbf{Acknowledgements.}  J. Li was partially supported by NSFC (No.11801090) and Z. Yin was partially supported by NSFC (No.11671407 and No.11271382),  FDCT (No. 098/2013/A3), Guangdong Special Support Program (No. 8-2015), and the key project of NSF of  Guangdong province (No. 2016A030311004).
%\vspace*{1em}

\end{document}